\documentclass{article}

\usepackage[utf8]{inputenc}
\usepackage[T2A]{fontenc}
\usepackage[english]{babel}
\usepackage[margin=1in]{geometry}
\usepackage{amsthm, amsfonts, amsmath, amssymb}
\usepackage{enumitem}

\DeclareMathOperator{\Ind}{Ind}
\DeclareMathOperator{\Nul}{Nul}
\DeclareMathOperator{\id}{id}

\let\phi\varphi
\let\epsilon\varepsilon
\newcommand{\g}{\mathbf g}
\newcommand{\h}{\mathbf h}
\newcommand{\e}{\mathbf e}
\newcommand{\Sph}{\mathbb S}
\newcommand{\Gr}{\mathrm{Gr}}
\newcommand{\R}{\mathbb R}
\newcommand{\Z}{\mathbb Z}

\renewcommand{\C}{\mathbb C}
\renewcommand{\Re}{\mathop{\mathrm{Re}}\nolimits}
\renewcommand{\Im}{\mathop{\mathrm{Im}}\nolimits}

\AtEndDocument{%
\par
\medskip
\begin{tabular}{@{}l@{}}%
\textsc{Higher School of Modern Mathematics MIPT}\\
\textsc{1 Klimentovskiy per., Moscow, Russia}\\
\textit{and}\\
\textsc{Faculty of Mathematics, National Research University Higher School of
Economics,}\\
\textsc{6 Usacheva Str., Moscow, Russia}\\
\textit{and}\\
\textsc{Independent University of Moscow,}\\
\textsc{11 Bolshoy Vlasyevskiy per., Moscow, Russia}\\
\textit{E-mail address}: \texttt{ea.morozov@hse.ru}
\end{tabular}}

\title{Index of Bipolar Surfaces to Otsuki Tori}

\author{Еgor Morozov}
\date{}

\theoremstyle{plain}
\newtheorem{theorem}{Theorem}[section]
\newtheorem{proposition}[theorem]{Proposition}
\newtheorem{claim}[theorem]{Claim}
\newtheorem{conjecture}{Conjecture}

\theoremstyle{remark}
\newtheorem{remark}[theorem]{Remark}

\begin{document}

\maketitle

\begin{abstract}
For each rational number $p/q\in (1/2,\sqrt 2/2)$ one can construct an
$\Sph^1$-equivariant minimal torus in $\Sph^3$ called Otsuki torus and denoted
by $O_{p/q}$.
The Lawson's bipolar surface construction applied to $O_{p/q}$ gives a minimal
torus $\widetilde O_{p/q}$ in $\Sph^4$.
In this paper we give upper and lower bounds on the Morse index and the nullity
of these tori for $p/q$ close to $\sqrt 2/2$.
We also state a numerically assisted conjecture concerning the
general case.
\end{abstract}

\section{Introduction}

In the present paper we study the index and the nullity of $\Sph^1$-equivariant
minimal tori in $\Sph^4$ called \emph{bipolar surfaces to Otsuki tori}.
These surfaces are obtained as a result of a two-step construction.
The first step is the construction of Otsuki tori, which are
$\Sph^1$-equivariant minimal tori in $\Sph^3$.
These tori were introduced by Otsuki in~\cite{otsuki1970minimal} but
later the original definition was significantly simplified by
Penskoi~\cite{penskoi2013extremal} with the help of Hsiang-Lawson construction
of equivariant minimal surfaces~\cite{hsiang1971minimal}.
In particular, it turns out that there is a natural bijection between Otsuki
tori and rational numbers from the interval $(1/2,\sqrt 2/2)$.
Following~\cite{penskoi2013extremal}, we denote by $O_{p/q}$ the Otsuki tori
corresponding to the number $p/q\in (1/2,\sqrt 2/2)$.

The second step is the application of Lawson's bipolar surface
construction~\cite[sec.~11]{lawson1970complete} to $O_{p/q}$.
Generally, starting from a minimal immersion
$u\colon\Sigma\looparrowright\Sph^3$, this construction gives a minimal
immersion $\tilde u\colon\Sigma\looparrowright\Sph^5$.
However, for some immersions $u$ the image of $\tilde u$ is contained in an
equatorial subsphere $\Sph^4\subset\Sph^5$ and this is the case for Otsuki tori.
Thus, we end up with a family of minimal surfaces in $\Sph^4$, which are
naturally called bipolar surfaces to Otsuki tori.
These surfaces are extensively studied in~\cite{karpukhin2014spectral}.
Following this paper, we denote by $\widetilde O_{p/q}$ the bipolar surface to
$O_{p/q}$, that is, the image of the corresponding bipolar immersion $\tilde u$.

All the surfaces $\widetilde O_{p/q}$ are tori as well.
However, if $q$ is even, then the corresponding bipolar immersion $\tilde u$
covers its own image twice, and this turns out to be very important in the
sequel.
In particular, it is convenient to introduce the surface $\widehat O_{p/q}$,
which is the two-sheeted cover of $\widetilde O_{p/q}$ if $q$ is even and
coincides with $\widetilde O_{p/q}$ otherwise.
In practice one makes all the computations for $\widehat O_{p/q}$ at first and
then descends on $\widetilde O_{p/q}$ if $q$ is even.

Initially the study of the surfaces $O_{p/q}$ and $\widetilde O_{p/q}$ was
motivated by spectral geometry. Recall that by a celebrated result of
Nadirashvili, El Soufi, and
Ilias~\cite{nadirashvili1996isoperimetric,elsoufi2008laplacian}, any metric $g$
induced by a minimal immersion of $\Sigma$ in a unit sphere is critical for the
functional $\bar\lambda_{N_\Sigma(2)}(g)$, where $\bar\lambda_k(g)$ is the
$k$-th normalized eigenvalue of the Laplace-Beltrami operator on $\Sigma$ and
$N_\Sigma(2)$ is the number of Laplace-Beltrami eigenvalues of $\Sigma$ less
than~2. However, it takes some work to compute $N_{O_{p/q}}(2)$ and
$N_{\widetilde O_{p/q}}(2)$. These computations are actually the main purpose of
the works~\cite{penskoi2013extremal} and~\cite{karpukhin2014spectral}. The main
result of~\cite{karpukhin2014spectral} states that
\begin{equation}\label{eq:intro-n2}
N_{\widetilde O_{p/q}}(2)=\begin{cases}
2q+4p-2,&\text{$q$ is odd;}\\
q+2p-2,&\text{$q$ is even.}
\end{cases}
\end{equation}
In fact, it is shown that $N_{\widehat O_{p/q}}(2)=2q+4p-2$ and then analyzed
which eigenfunctions descend on $\widetilde O_{p/q}$ if $q$ is even. Critical
metrics (also called extremal metrics) are studied
in~\cite{karpukhin2015spectral,penskoi2015generalized}, see also
reviews~\cite{penskoi2013metrics,penskoi2019isoperimetric}.

Other important geometric quantities associated with a minimal surface are the
(Morse) index and the nullity.
The index of $O_{p/q}$ is computed in~\cite{morozov2023index}.
A general fact proved there implies that
$$
\Ind(O_{p/q})=N_{\widehat O_{p/q}}(2)=2q+4p-2.
$$
The aim of the current paper is to obtain estimates on the index and the nullity
of the surfaces $\widetilde O_{p/q}$.
Although this turns out to be a much more complicated problem, a rough upper
bound can be obtained as follows.
It follows from~\cite[Theorem~1.1]{ejiri2008comparison}
and~\cite[Proposition~1.6]{karpukhin2021index} that
\begin{equation}\label{eq:intro-indub}
\Ind\tilde O_{p/q}\leqslant 5N_{\widetilde O_{p/q}}(2)+2.
\end{equation}
Combining this with~\eqref{eq:intro-n2}, we obtain
$$
\Ind\widetilde O_{p/q}\leqslant\begin{cases}
10q+20p-8,&\text{$q$ is odd;}\\
5q+10p-8,&\text{$q$ is even.}
\end{cases}
$$
Our main result is that if $p/q$ is sufficiently close to $\sqrt 2/2$ (speaking
informally, this means that $\widetilde O_{p/q}$ is ``close to the Clifford
torus''), then one can improve the bound~\eqref{eq:intro-indub} and also obtain
a lower bound.

\begin{theorem}\label{th:intro-main}
There exists $\epsilon>0$ such that if
$\frac{\sqrt{2}}{2}-\frac{p}{q}<\epsilon$, then the following inequalities hold
\begin{align*}
6q+8p-3\leqslant&\Ind\widetilde O_{p/q}\leqslant 10q+4p-5,&&\text{$q$ is odd,}\\
3q+4p-3\leqslant&\Ind\widetilde O_{p/q}\leqslant 5q+2p-5,&&\text{$q$ is even,}\\
9\leqslant&\Nul\widetilde O_{p/q}\leqslant 13.
\end{align*}
\end{theorem}

We expect that the computation of the exact value of $\Ind(\widetilde O_{p/q})$
(at least, for all $p/q\in (1/2,\sqrt 2/2)$) is a very complicated problem.
See Remark~\ref{rem:l12-disc} for a detailed discussion of the difficulties
encountered and a conjecture regarding the general case.

The plan of the proof of Theorem~\ref{th:intro-main} is the following.
Recall that the index can be defined as the number of negative eigenvalues of
the Jacobi stability operator $L$.
First we compute the operator $L$ on $\widehat O_{p/q}$ in appropriate local
coordinates.
Then, using separation of variables, we reduce the counting of negative
eigenvalues of $L$ to the counting of negative eigenvalues of a family of matrix
Sturm-Liouville operators $S_l$, depending on a non-negative integer parameter
$l$.
For $l=0$ the matrix Sturm-Liouville problem reduces to a pair of scalar
Sturm-Liouville problems.
We then apply the technique based on the Sturm Oscillation Theorem to determine
the number of negative eigenvalues.
This is exactly the method used
in~\cite{penskoi2013extremal,karpukhin2014spectral} for
computing $N_{O_{p/q}}(2)$ and $N_{\widetilde O_{p/q}}(2)$ and previously
applied for Lawson tau-surfaces~\cite{penskoi2012extremal}.

Unfortunately, for $l>0$ the matrix Sturm-Liouville problem does not split into
scalar problems and we cannot apply Sturm Oscillation Theorem in this case.
Instead we look at the limit case $p/q=\sqrt 2/2$, which is a Clifford torus in
equatorial $\Sph^3\subset\Sph^5$.
In this case all eigenvalues can be computed explicitly and we use this to get
the desired bounds.
A price to pay is that the estimates hold only for $p/q$ sufficiently close to
$\sqrt 2/2$.
In particular, we are not able even to point out some specific $p/q$ for which
our estimates hold.
However we hope that the proposed method provides some intuition and can be
easily adopted for numerical computations.
For example, we are able to compute numerically the index of $\widetilde
O_{2/3}$ (see Conjecture~\ref{con:l12-disc}).

The paper is organised as follows.
In section~\ref{sec:notdef} we fix the notation and recall the basic
definitions.
In section~\ref{sec:bip} we describe a parametrization of~$\widehat O_{p/q}$
used throughout the paper.
In section~\ref{sec:jac} we compute the Jacobi stability operator on $\widehat
O_{p/q}$ together with some of its eigensections.
In section~\ref{sec:sep} we separate variables and introduce a family of matrix
Sturm-Liouville operators $S_l$.
In section~\ref{sec:l0} we analyse the ``easy'' case $l=0$ and in
section~\ref{sec:l12} we analyse the ``difficult'' cases $l=1,2$ and prove
Theorem~\ref{th:intro-main}.
The paper is concluded with Remark~\ref{rem:l12-disc} containing a short
discussion of the result.

\section*{Acknowledgments}

The author is partially supported by the Theoretical Physics and Mathematics
Advancement Foundation ``BASIS'' grant Leader (Math) 21-7-1-45-1 and Simons-IUM
Fellowship. An earlier version of this text was presented on the 26th
All-Russian August M\"obius Contest.

The author is grateful to A.\,Penskoi for drawing attention to the problem.
Also thanks to V.\,Medvedev and M.\,Karpukhin for useful discussions.

\section{Preliminaries}

\subsection{Notation and definitions}\label{sec:notdef}

In this section we fix the notation and recall the definition of the Morse index
and the nullity of a minimal surface.

Let $\Sigma$ be an oriented minimal surface in the unit 4-dimensional sphere
$\Sph^4$ with the standard round metric $\langle\cdot\,,\cdot\rangle$.
Throughout this paper $T\Sigma$ and $N\Sigma$ denote the tangent bundle of
$\Sigma$ and the normal bundle to $\Sigma$ in $\R^5$ respectively.
For any vector $v\in T\Sph^4$ let $v^\top$ and $v^\bot$ denote the orthogonal
projections of $v$ on $T\Sigma$ and $N\Sigma$ respectively.
The Levi-Civita connection in $\Sph^n$ is denoted by $\nabla$ and the induced
connections in $T\Sigma$ and $N\Sigma$ are denoted by $\nabla^\top$ and
$\nabla^\bot$ respectively.
The space of all smooth sections of a vector bundle $E$ is denoted by
$\Gamma(E)$.

Let $e_1,e_2$ be a local orthonormal basis in $T\Sigma$.
Then the Laplace-Beltrami operator on $\Sigma$ is given by
$$
\Delta f=\sum_{i=1}^2 ((\nabla_{e_i}^\top e_i)f-e_i(e_i f)),\quad
f\in C^\infty(\Sigma),
$$
the Laplace-Beltrami operator in the normal bundle
$\Delta^\bot\colon\Gamma(N\Sigma)\to\Gamma(N\Sigma)$ is given by
\begin{equation}\label{eq:intro-deltabot}
\Delta^\bot X=\sum_{i=1}^2 (\nabla_{\nabla_{e_i}^\top e_i}^\bot X-\nabla_{e_i}^\bot\nabla_{e_i}^\bot X),\quad
X\in\Gamma(N\Sigma),
\end{equation}
and the Simons operator $\mathcal B\colon\Gamma(N\Sigma)\to\Gamma(N\Sigma)$ is
given by
$$
\mathcal B(X)=\sum_{i=1}^2 \langle b_{ij},X\rangle b_{ij},\quad
X\in\Gamma(N\Sigma),
$$
where $b_{ij}:=B(e_i,e_j)$ and
$$
B(X,Y)=(\nabla_X Y)^\bot,\quad
X,Y\in\Gamma(T\Sigma)
$$
is the second fundamental form of $\Sigma$ in $\Sph^4$.
Finally, the Jacobi stability operator
$L\colon\Gamma(N\Sigma)\to\Gamma(N\Sigma)$ is given by
$$
L=\Delta^\bot-2-\mathcal B.
$$
It is well-known that the operator $L$ is elliptic. In particular, the spectrum
of $L$ is discrete and has the form
$$
\lambda_1\leqslant\lambda_2\leqslant\dots\leqslant\lambda_n\leqslant\dots\nearrow+\infty,
$$
where each eigenvalue is listed as many times as its multiplicity is. Then the
quantities
$$
\Ind\Sigma=\#\{\lambda_k<0\}\quad\text{and}\quad
\Nul\Sigma=\#\{\lambda_k=0\}
$$
are called respectively the (Morse) \emph{index} and the \emph{nullity} of the
minimal surface $\Sigma$.
Hereafter we adhere to the following conventions concerning the collections of
eigenvalues:
1) all collections of eigenvalues are considered as \emph{multisets}, i.e., each
eigenvalue is counted with multiplicity;
2) the range of $k$ is maximal possible (here, for instance,
$k\geqslant 1$).

\subsection{Bipolar surfaces to Otsuki tori}\label{sec:bip}

In this section we shortly describe a convenient parametrization of $\widetilde
O_{p/q}$ from~\cite{karpukhin2014spectral}.
Here we give only the most necessary definitions.
In particular, we do not describe neither Hsiang-Lawson construction nor
Lawson's bipolar surface construction.
We refer to~\cite[sec.~2.4]{karpukhin2014spectral} for the details.

Fix $b\in(-\frac{\pi}2,0)$ and define the function $t(\phi)$ on $[b,-b]$ by
$$
t(\phi):=\int_b^\phi \frac{2\pi\cos^3\psi\,d\psi}{\sqrt{\cos^4\psi-\cos^4 b}}.
$$
Note that the singularities at $\psi=\pm b$ are integrable so that $t(\phi)$ is
well-defined and increasing. Let $T=t(-b)$ and define the function $\phi(t)$ on
$[0,T]$ as the inverse of $t(\phi)$.
Then the function $\phi(t)$ satisfies
$$
\phi(0)=b,\quad
\dot\phi(0)=0,\quad
\phi(T-t)=-\phi(t).
$$
In particular, $\phi(t)$ extends to a smooth $T$-antiperiodic function on $\R$
satisfying
\begin{equation}\label{eq:ots-dotphi}
\dot\phi(t)^2=\frac{\cos^4\phi(t)-\cos^4 b}{4\pi^2\cos^6\phi(t)}.
\end{equation}
Define the function $\theta(t)$ by
\begin{equation}\label{eq:ots-dottheta}
\dot\theta(t)=\frac{\cos^2 b}{2\pi\cos^4\phi(t)},\quad
\theta(0)=0.
\end{equation}
\begin{proposition}
1) For each $b\in (-\frac{\pi}2,0)$ the image of the immersion $u_b\colon
(\R/2\pi\Z)\times\R\looparrowright\R^5$ given by
\begin{equation}\label{eq:ots-param}
u_b(\alpha,t)=
\begin{pmatrix}
\cos\alpha\cos\phi(t)\sin\theta(t)\\
\sin\alpha\cos\phi(t)\sin\theta(t)\\
\cos\alpha\cos\phi(t)\cos\theta(t)\\
\sin\alpha\cos\phi(t)\cos\theta(t)\\
\sin\phi(t)
\end{pmatrix}
\end{equation}
is a minimal surface in $\Sph^4$.

2) The map $u_b$ is periodic in $t$ if and only if the number
$$
\Xi(b):=\int_b^{-b} \frac{\cos^2 b\,d\phi}{\cos\phi\sqrt{\cos^4\phi-\cos^4 b}}
$$
is a rational multiple of $\pi$.
More precisely, if $\Xi(b)=(p/q)\pi$, where $\mathrm{gcd}(p,q)=1$, then the map
$u_b$ is $t_0$-periodic with $t_0=2qT$.
The number $p/q$ can be any rational number from the interval
$(1/2,\sqrt{2}/2)$.

3) In the conditions of 2), if $q$ is even, then the immersion $u_b$ is
invariant under the map $(\alpha,t)\mapsto (\alpha+\pi,t+\frac{t_0}2)$.
\end{proposition}

One can easily obtain the proof of this proposition from~\cite[sec.~2.4
and~2.5]{karpukhin2014spectral}.
In this paper we define $\widetilde O_{p/q}$ to be the image of $u_b$, where
$\Xi(b)=(p/q)\pi$.
Let $\widehat O_{p/q}$ be the orientable two-sheeted cover of $\widetilde
O_{p/q}$ if $q$ is even and coincide with $\widetilde O_{p/q}$ if $q$ is odd.

\section{Proof of Theorem~\ref{th:intro-main}}

\subsection{Jacobi operator on $\widehat O_{p/q}$}\label{sec:jac}

In this section we compute the Jacobi stability operator $L$ on $\widehat
O_{p/q}$ in local coordinates $\alpha,t$.
We put $u=u_b$ and suppress the argument $t$ in $\phi(t)$ and $\theta(t)$ for
simplicity.

\begin{proposition}\label{pr:jac-basis}
For any point $x\in\widehat O_{p/q}$ the following vectors form an orthonormal
basis of $T_x\R^5$,
\begin{equation}\label{eq:jac-basis}
\begin{gathered}
N=u=\begin{pmatrix}
\cos\alpha\cos\phi\sin\theta\\
\sin\alpha\cos\phi\sin\theta\\
\cos\alpha\cos\phi\cos\theta\\
\sin\alpha\cos\phi\cos\theta\\
\sin\phi
\end{pmatrix},\\
e_1=\frac{\partial_\alpha u}{\cos\phi}=\begin{pmatrix}
-\sin\alpha\sin\theta\\
\cos\alpha\sin\theta\\
-\sin\alpha\cos\theta\\
\cos\alpha\cos\theta\\
0
\end{pmatrix},
e_2=2\pi\cos\phi\,\partial_t u=2\pi\cos\phi\begin{pmatrix}
\cos\alpha(-\sin\phi\sin\theta\,\dot\phi+\cos\phi\cos\theta\,\dot\theta)\\
\sin\alpha(-\sin\phi\sin\theta\,\dot\phi+\cos\phi\cos\theta\,\dot\theta)\\
\cos\alpha(-\sin\phi\cos\theta\,\dot\phi-\cos\phi\sin\theta\,\dot\theta)\\
\sin\alpha(-\sin\phi\cos\theta\,\dot\phi-\cos\phi\sin\theta\,\dot\theta)\\
\cos\phi\,\dot\phi
\end{pmatrix},\\
n_1=\begin{pmatrix}
\sin\alpha\cos\theta\\
-\cos\alpha\cos\theta\\
-\sin\alpha\sin\theta\\
\cos\alpha\sin\theta\\
0
\end{pmatrix},
n_2=2\pi\cos\phi\begin{pmatrix}
-\cos\alpha(\cos\theta\,\dot\phi+\sin\theta\sin\phi\cos\phi\,\dot\theta)\\
-\sin\alpha(\cos\theta\,\dot\phi+\sin\theta\sin\phi\cos\phi\,\dot\theta)\\
\cos\alpha(\sin\theta\,\dot\phi-\cos\theta\sin\phi\cos\phi\,\dot\theta)\\
\sin\alpha(\sin\theta\,\dot\phi-\cos\theta\sin\phi\cos\phi\,\dot\theta)\\
\cos^2\phi\,\dot\theta
\end{pmatrix}.
\end{gathered}
\end{equation}
Moreover, $e_1,e_2$ is a basis of $T_x\widehat O_{p/q}$ and $n_1,n_2$ is a basis
of $N_x\widehat O_{p/q}$.
If $q$ is even, then both $n_1$ and $n_2$ descend on
$N\widetilde O_{p/q}$.
\end{proposition}

The proof of this proposition is a direct verification.
In the sequel, all computations are made w.r.t. the basis~\eqref{eq:jac-basis}.
The (local) section $f_1 n_1+f_2 n_2$ of $N\widehat O_{p/q}$ is denoted by
$\left[\begin{smallmatrix}f_1\\ f_2\end{smallmatrix}\right]$.

\begin{proposition}\label{pr:jac-simons}
The matrix of the Simons operator $\mathcal B\colon\Gamma(N\widehat
O_{p/q})\to\Gamma(N\widehat O_{p/q})$ in the basis $n_1,n_2$ is given by
$$
\mathcal B=\begin{bmatrix}
8\pi^2\cos^2\phi\,\dot\theta^2 & 0\\
0 & 8\pi^2\sin^2\phi\cos^2\phi\,\dot\theta^2
\end{bmatrix}.
$$
\end{proposition}

\begin{proof}
The entry $\mathcal B_{\mu\nu}$ of the matrix $\mathcal B$ equals
$\sum_{i,j=1}^2 \langle b_{ij},n_\mu\rangle\langle b_{ij},n_\nu\rangle$.
We have
$$
b_{11}=(\nabla_{e_1}e_1)^\bot=\frac{(\partial_\alpha^2 u)^\bot}{\cos^2\phi},\quad
b_{12}=b_{21}=(\nabla_{e_1}e_2)^\bot=2\pi (\partial_\alpha\partial_t u)^\bot,
$$
and $b_{22}=-b_{11}$ by the minimality of $u$. Hence,
$$
\langle b_{11},n_\mu\rangle=-\langle b_{22},n_\mu\rangle=\frac{1}{\cos^2\phi}\langle\partial_\alpha^2 u,n_\mu\rangle,\quad
\langle b_{12},n_\mu\rangle=\langle b_{21},n_\mu\rangle=2\pi\langle\partial_\alpha\partial_t u,n_\mu\rangle,\quad\mu=1,2,
$$
and the rest is a direct computation using~\eqref{eq:jac-basis}.
\end{proof}

\begin{proposition}
The operator $\Delta^\bot\colon\Gamma(N\widehat O_{p/q})\to\Gamma(N\widehat
O_{p/q})$ in the basis $n_1,n_2$ has the form
$$
\Delta^\bot\begin{bmatrix} f_1\\ f_2\end{bmatrix}=
\begin{bmatrix}
\Delta f_1+4\pi^2\dot\phi^2 f_1-\frac{4\pi\dot\phi}{\cos\phi}\partial_\alpha f_2,\\
\Delta f_2+4\pi^2\dot\phi^2 f_2+\frac{4\pi\dot\phi}{\cos\phi}\partial_\alpha f_1.
\end{bmatrix}.
$$
\end{proposition}

\begin{proof}
It easily follows from~\eqref{eq:intro-deltabot} that for any
$n\in\Gamma(N\widehat O_{p/q})$ and $f\in C^\infty(\widehat O_{p/q})$ we have
\begin{equation}\label{eq:jac-deltafn}
\Delta^\bot(fn)=
f\Delta^\bot n+(\Delta f)n-2\sum_{i=1}^2 (e_i f)\nabla_{e_i}^\bot n.
\end{equation}
Therefore it suffices to calculate $\Delta^\bot n_\mu$ and $\nabla_{e_i}^\bot
n_\mu$ for $i,\mu=1,2$. We have
\begin{equation}\label{eq:jac-partialen}
\partial_{e_1}n_1=\frac{1}{\cos\phi}\partial_\alpha n_1=\frac{1}{\cos\phi}\begin{pmatrix}
\cos\alpha\cos\theta\\
\sin\alpha\cos\theta\\
-\cos\alpha\sin\theta\\
-\sin\alpha\sin\theta\\
0
\end{pmatrix},\quad
\partial_{e_2}n_1=2\pi\cos\phi\,\partial_t n_1=
2\pi\cos\phi\begin{pmatrix}
-\sin\alpha\sin\theta\,\dot\theta\\
\cos\alpha\sin\theta\,\dot\theta\\
-\sin\alpha\cos\theta\,\dot\theta\\
\cos\alpha\cos\theta\,\dot\theta\\
0
\end{pmatrix}.
\end{equation}
From this it is easy to see that
$$
\langle\partial_{e_1}n_2,n_1\rangle=-\langle\partial_{e_1}n_1,n_2\rangle=2\pi\dot\phi\quad\text{and}\quad
\langle\partial_{e_i}n_\mu,n_\nu\rangle=0
\text{ for all other $i,\mu,\nu=1,2$.}
$$
Hence,
\begin{equation}\label{eq:jac-nablaen}
\begin{gathered}
\nabla_{e_1}^\bot n_1=\langle\partial_{e_1}n_1,n_2\rangle n_2=-2\pi\dot\phi\,n_2,\quad
\nabla_{e_1}^\bot n_2=\langle\partial_{e_1}n_2,n_1\rangle n_1=2\pi\dot\phi\,n_1,\\
\nabla_{e_2}^\bot n_1=\nabla_{e_2}^\bot n_2=0.
\end{gathered}
\end{equation}
Further, using~\eqref{eq:jac-basis}, we get
\begin{equation}\label{eq:jac-nablaee}
\nabla_{e_1}^\top e_1=\langle\partial_{e_1}e_1,e_2\rangle e_2=2\pi\sin\phi\,\dot\phi\,e_2,\quad
\nabla_{e_2}^\top e_2=-\langle\partial_{e_2}e_1,e_2\rangle e_1=0,
\end{equation}
and using~\eqref{eq:intro-deltabot}, we find
\begin{equation}\label{eq:jac-deltan}
\Delta^\bot n_1=-\nabla_{e_1}^\bot\nabla_{e_1}^\bot n_1=4\pi^2\dot\phi^2 n_1,\quad
\Delta^\bot n_2=-\nabla_{e_1}^\bot\nabla_{e_1}^\bot n_2=4\pi^2\dot\phi^2 n_2,
\end{equation}
where all other terms in~\eqref{eq:intro-deltabot} vanish because
of~\eqref{eq:jac-nablaen} and~\eqref{eq:jac-nablaee}.
The proposition now follows
from~\eqref{eq:jac-deltafn},~\eqref{eq:jac-nablaen},~\eqref{eq:jac-deltan}.
\end{proof}

\begin{proposition}[{\cite[Proposition~6]{karpukhin2014spectral}}]\label{pr:jac-laplace}
The Laplace-Beltrami operator on $\widehat O_{p/q}$ is given by the formula
$$
\Delta f=-\frac{1}{\cos^2\phi}\partial_\alpha^2 f-\partial_t(4\pi^2\cos^2\phi\,\partial_t f).
$$
\end{proposition}

\begin{proposition}\label{pr:jac-jac}
The Jacobi stability operator $L\colon\Gamma(N\widehat
O_{p/q})\to\Gamma(N\widehat O_{p/q})$ in the basis $n_1,n_2$ has the form
$$
L\begin{bmatrix} f_1\\ f_2\end{bmatrix}=
\begin{bmatrix}
\Delta f_1-4\pi^2\dot\phi^2 f_1-\frac{4\pi\dot\phi}{\cos\phi}\partial_\alpha f_2-2f_1-8\pi^2\cos^2\phi\,\dot\theta^2 f_1\\
\Delta f_2-4\pi^2\dot\phi^2 f_2+\frac{4\pi\dot\phi}{\cos\phi}\partial_\alpha f_1-2f_2-8\pi^2\cos^2\phi\sin^2\phi\,\dot\theta^2 f_2
\end{bmatrix}.
$$
\end{proposition}

\begin{proof}
This follows from Propositions~\ref{pr:jac-simons}--\ref{pr:jac-laplace}.
\end{proof}

Some eigensections of $L$ can be found from geometrical considerations.
In the sequel, the following proposition turns out to be very useful.

\begin{proposition}\label{pr:esec}
(i) The following sections of $N\widehat O_{p/q}$ are in the kernel of $L$:
\begin{gather*}
\begin{bmatrix}
\cos\phi\sin 2\theta\\ 0
\end{bmatrix},
\begin{bmatrix}
\cos\phi\cos 2\theta\\ 0
\end{bmatrix},
\begin{bmatrix}
0\\ 2\pi\cos^2\phi\,\dot\phi
\end{bmatrix},
\begin{bmatrix}
-\sin\alpha\sin\phi\cos\theta\\
2\pi\cos\alpha\cos\phi(\sin\phi\cos\theta\,\dot\phi+\sin\theta\cos\phi\,\dot\theta)
\end{bmatrix},\\
\begin{bmatrix}
\cos\alpha\sin\phi\cos\theta\\
2\pi\sin\alpha\cos\phi(\sin\phi\cos\theta\,\dot\phi+\sin\theta\cos\phi\,\dot\theta)
\end{bmatrix},
\begin{bmatrix}
-\sin\alpha\sin\phi\sin\theta\\
2\pi\cos\alpha\cos\phi(\sin\phi\sin\theta\,\dot\phi-\cos\theta\cos\phi\,\dot\theta)
\end{bmatrix},\\
\begin{bmatrix}
\cos\alpha\sin\phi\sin\theta\\
2\pi\sin\alpha\cos\phi(\sin\phi\sin\theta\,\dot\phi-\cos\theta\cos\phi\,\dot\theta)
\end{bmatrix},
\begin{bmatrix}
-\sin 2\alpha\cos\phi\\
2\pi\cos 2\alpha\cos^2\phi\,\dot\phi
\end{bmatrix},
\begin{bmatrix}
\cos 2\alpha\cos\phi\\
2\pi\sin 2\alpha\cos^2\phi\,\dot\phi
\end{bmatrix}.
\end{gather*}

(ii) The following sections of $N\widehat O_{p/q}$ are eigensections of $L$ with
eigenvalue $-2$:
\begin{gather*}
\begin{bmatrix}
-\sin\alpha\cos\theta\\
\cos\alpha(\cos\theta\,\dot\phi+\sin\theta\sin\phi\cos\phi\,\dot\theta)
\end{bmatrix},
\begin{bmatrix}
\cos\alpha\cos\theta\\
\sin\alpha(\cos\theta\,\dot\phi+\sin\theta\sin\phi\cos\phi\,\dot\theta)
\end{bmatrix},\\
\begin{bmatrix}
-\sin\alpha\sin\theta\\
\cos\alpha(\sin\theta\,\dot\phi-\cos\theta\sin\phi\cos\phi\,\dot\theta)
\end{bmatrix},
\begin{bmatrix}
\cos\alpha\sin\theta\\
\sin\alpha(\sin\theta\,\dot\phi-\cos\theta\sin\phi\cos\phi\,\dot\theta)
\end{bmatrix},
\begin{bmatrix}
0\\
\cos^3\phi\,\dot\theta
\end{bmatrix}.
\end{gather*}
\end{proposition}

\begin{proof}
1) It follows from~\cite[Lemma~5.1.7]{simons1968minimal} that the image of any
Killing vector field on $\Sph^4$ under the orthogonal projection on $N\widehat
O_{p/q}$ belongs to the kernel of $L$.
The space of Killing vector fields on $\Sph^4$ is spanned by the fields
$x^j\frac{\partial}{\partial x^i}-x^i\frac{\partial}{\partial x^j}$, where
$i,j=1,\ldots,5$.
The rest of the proof is a direct computation using~\eqref{eq:jac-basis}.

2) It follows from~\cite[Lemma~5.1.4]{simons1968minimal} that the image of any
constant vector field on $\R^5$ under the orthogonal projection on $N\widehat
O_{p/q}$ is an eigensection of $L$ with eigenvalue $-2$.
The rest of the proof is again a direct computation using~\eqref{eq:jac-basis}.
\end{proof}

\subsection{Separation of variables}\label{sec:sep}

In this section we reduce the spectral problem for $L$ to a family of matrix
Sturm-Liouville eigenvalue problems.

\begin{proposition}\label{pr:sep}
For each $l=0,1,2,\ldots$ consider the following matrix Sturm-Liouville
eigenvalue problem on $[0,t_0]$ with periodic boundary conditions
\begin{equation}\label{eq:sep-sl}
\begin{cases}
S_l\h=\lambda\h,\\
\h(0)=\h(t_0),\quad \h'(0)=\h'(t_0),
\end{cases}
\end{equation}
where $\h(t)=\left(\begin{smallmatrix}h_1(t)\\ h_2(t)\end{smallmatrix}\right)$
is a vector function and $S_l=-\partial_t(p(t)\partial_t)+A(l,t)$ is a matrix
Sturm-Liouville operator with
\begin{gather*}
p(t)=4\pi^2\cos^2\phi,\\
A(l,t)=\begin{pmatrix}
\frac{l^2}{\cos^2\phi}+4\pi^2\dot\phi^2-2-8\pi^2\cos^2\phi\,\dot\theta^2 & -\frac{4\pi l\dot\phi}{\cos\phi}\\
-\frac{4\pi l\dot\phi}{\cos\phi} & \frac{l^2}{\cos^2\phi}+4\pi^2\dot\phi^2-2-8\pi^2\sin^2\phi\cos^2\phi\,\dot\theta^2
\end{pmatrix}.
\end{gather*}
Then the $\lambda$-eigenspace of the operator $L$ has a basis consisting of the
eigensections of the form
\begin{equation}\label{eq:sep-eigensec}
\begin{bmatrix}
h_1(t)\cos l\alpha\\
h_2(t)\sin l\alpha
\end{bmatrix}\quad\text{and}\quad
\begin{bmatrix}
-h_1(t)\sin l\alpha\\
h_2(t)\cos l\alpha
\end{bmatrix},
\end{equation}
where $\h(t)=\left(\begin{smallmatrix}h_1(t)\\ h_2(t)\end{smallmatrix}\right)$
solves the problem~\eqref{eq:sep-sl} and $l\geqslant 0$ is integer.
Moreover, if $q$ is even, then one can choose this basis of eigensections in
such a way that for each eigensection the corresponding vector eigenfunction
$\h(t)$ of~\eqref{eq:sep-sl} is either $\frac{t_0}2$-periodic or
$\frac{t_0}2$-antiperiodic.
In this case, the eigensections~\eqref{eq:sep-eigensec} descend on $\widetilde
O_{p/q}$ exactly when either $\h(t)$ is $\frac{t_0}2$-periodic and $l$ is even
or $\h(t)$ is $\frac{t_0}2$-antiperiodic and $l$ is odd.
\end{proposition}

\begin{proof}
Since the operator $L$ commutes with $\partial_\alpha$, we see that the
$\lambda$-eigenspace of the operator $L$ has a basis such that each basis
eigensection has the form~\eqref{eq:sep-eigensec}, where $\h(t)$ is a vector
function and $l\geqslant 0$ is integer. Then a direct computation involving
Proposition~\ref{pr:jac-jac} shows that $\h(t)$ solves~\eqref{eq:sep-sl}.

If $q$ is even, then the coefficients of the operator $S_l$ are
$\frac{t_0}2$-periodic.
In other words, the operator $S_l$ commutes with the map $\iota\colon
\h(t)\mapsto \h\left(t+\frac{t_0}2\right)$, where $\h(t)$ is considered as a
vector function on the circle $\R/t_0\Z$.
Hence there exists a joint basis of eigenfunctions for~$S_l$ and $\iota$.
Since $\iota^2=\id$, we obtain that each $\iota$-eigenvalue is $\pm 1$, which
means that the corresponding vector eigenfunction $\h(t)$ is either
$\frac{t_0}2$-periodic or $\frac{t_0}2$-antiperiodic.
For the last claim it suffices to remark that a section
$\left[\begin{smallmatrix}f_1\\f_2\end{smallmatrix}\right]\in\Gamma(N\widehat
O_{p/q})$ descends on $\widetilde O_{p/q}$ exactly when both $f_1$ and $f_2$ are
invariant under the map $(\alpha,t)\mapsto (\alpha+\pi,t+\frac{t_0}2)$ and apply
the last claim of Proposition~\ref{pr:jac-basis}.
\end{proof}

Let
$$
\lambda_1(l)\leqslant\lambda_2(l)\leqslant\ldots\leqslant\lambda_k(l)\leqslant\ldots\nearrow+\infty
$$
be the eigenvalues of the problem~\eqref{eq:sep-sl}.
If $q$ is even, then each eigenvalue is one of two types, depending on whether
the corresponding vector eigenfunction is $\frac{t_0}2$-periodic or
$\frac{t_0}2$-antiperiodic.
We call them \emph{periodic} and \emph{antiperiodic} eigenvalues respectively.
Let
$$
\lambda_{i_1}(l)\leqslant\lambda_{i_2}(l)\leqslant\ldots\leqslant\lambda_{i_k}(l)\leqslant\ldots\nearrow+\infty
\quad\text{and}\quad
\lambda_{j_1}(l)\leqslant\lambda_{j_2}(l)\leqslant\ldots\leqslant\lambda_{j_k}(l)\leqslant\ldots\nearrow+\infty
$$
be all the periodic and antiperiodic eigenvalues of~\eqref{eq:sep-sl}
respectively. In particular,
$$
\{i_1,i_2,\ldots\}\sqcup\{j_1,j_2,\ldots\}=\mathbb N.
$$
Put $\lambda_k^+(l):=\lambda_{i_k}(l)$ and $\lambda_k^-(l):=\lambda_{j_k}(l)$
for each $k\geqslant 1$.

\begin{proposition}\label{pr:sep-ind}
If $q$ is odd, then
\begin{equation}\label{eq:sep-indqodd}
\Ind(\widetilde O_{p/q})=\#\{\lambda_k(0)<0\}+2\sum_{l=1}^\infty\#\{\lambda_k(l)<0\},\quad
\Nul(\widetilde O_{p/q})=\#\{\lambda_k(0)=0\}+2\sum_{l=1}^\infty\#\{\lambda_k(l)=0\},
\end{equation}
and if $q$ is even, then
\begin{equation}\label{eq:sep-indqeven}
\Ind(\widetilde O_{p/q})=\#\{\lambda_k^+(0)<0\}+2\sum_{l=1}^\infty\#\{\lambda_k^{(-1)^l}(l)<0\},\quad
\Nul(\widetilde O_{p/q})=\#\{\lambda_k^+(0)=0\}+2\sum_{l=1}^\infty\#\{\lambda_k^{(-1)^l}(l)=0\}.
\end{equation}
\end{proposition}

\begin{proof}
Let $q$ be odd. If $l>0$, then for each eigenfunction of~\eqref{eq:sep-sl} the
corresponding eigensections~\eqref{eq:sep-eigensec} are linearly independent.
Hence, in this case each negative (respectively, zero) eigenvalue
$\lambda_k(l)<0$ contributes~2 to $\Ind(\widetilde O_{p/q})$ (respectively,
$\Nul(\widetilde O_{p/q})$).
If $l=0$, then the matrix equation~\eqref{eq:sep-sl} separates into two scalar
Sturm-Liouville equations and each negative (respectively, zero) eigenvalue of
each of these equations contributes~1 to $\Ind(\widetilde O_{p/q})$
(respectively, $\Nul(\widetilde O_{p/q}$)).
This proves~\eqref{eq:sep-indqodd}.

Let $q$ be even. Then by the last point of Proposition~\ref{pr:sep}, if $l$ is
even, then the eigensections~\eqref{eq:sep-eigensec} both descend on $\widetilde
O_{p/q}$ exactly when the corresponding eigenvalue $\lambda_k(l)$ is periodic,
and if $l$ is odd, then the same happens exactly when the corresponding
eigenvalue is antiperiodic.
This proves~\eqref{eq:sep-indqeven}.
\end{proof}

\begin{proposition}\label{pr:sep-l3}
For $l\geqslant 3$ the inequality $\lambda_1(l)>0$ holds.
\end{proposition}

\begin{proof}
By the variational characterization of eigenvalues for the
problem~\eqref{eq:sep-sl} we have
$$
\lambda_1(l)=\inf_{\h\in C^1(\R/t_0\Z,\R^2)}\frac{\int_0^{t_0} (p(t)|\h'(t)|^2+\langle \h(t),A(l,t)\h(t)\rangle)\,dt}{\int_0^{t_0} |\h(t)|^2\,dt}\geqslant
\inf_{\h\in C^1(\R/t_0\Z,\R^2)}\frac{\int_0^{t_0} \langle \h(t),A(l,t)\h(t)\rangle\,dt}{\int_0^{t_0} |\h(t)|^2\,dt},
$$
Hence, it suffices to show that if $l\geqslant 3$, then the matrix $A(l,t)$ is
positive definite for each $t$. This holds as soon as the inequality
$$
\frac{l^2}{\cos^2\phi}+4\pi^2\dot\phi^2-2-8\pi^2\cos^2\phi\,\dot\theta^2>
\left|\frac{4\pi l\dot\phi}{\cos\phi}\right|
$$
holds for each $t$.
Using~\eqref{eq:ots-dotphi} and~\eqref{eq:ots-dottheta}, one can rewrite this
inequality in the form
$$
\left(l-\sqrt{1-\frac{\cos^4 b}{\cos^4\phi}}\right)^2>
2\left(\cos^2\phi+\frac{\cos^4 b}{\cos^4\phi}\right),
$$
which holds for $l\geqslant 3$ because in this case we have
$\text{l.h.s.}\geqslant 4>\text{r.h.s.}$
\end{proof}

\subsection{Case $l=0$}\label{sec:l0}

This case is easy because the matrix problem~\eqref{eq:sep-sl} splits into two
scalar Sturm-Liouville problems
\begin{equation}\label{eq:l0-h1}
\begin{cases}
-(4\pi^2\cos^2\phi\,h_1')'+(4\pi^2\dot\phi^2-2-8\pi^2\cos^2\phi\,\dot\theta^2)h_1=\lambda h_1,\\
h_1(t_0)=h_1(0),\quad h_1'(t_0)=h_1'(0).
\end{cases}
\end{equation}
and
\begin{equation}\label{eq:l0-h2}
\begin{cases}
-(4\pi^2\cos^2\phi\,h_2')'+(4\pi^2\dot\phi^2-2-8\pi^2\sin^2\phi\cos^2\phi\,\dot\theta^2)h_2=\lambda h_2,\\
h_2(t_0)=h_2(0),\quad h_2'(t_0)=h_2'(0).
\end{cases}
\end{equation}
Our aim is the following
\begin{proposition}\label{pr:l0}
We have
$$
\#\{\lambda_k(0)<0\}=2q+4p-1,\quad
\#\{\lambda_k(0)=0\}=3.
$$
If $q$ is even, then, in addition, we have
$$
\#\{\lambda_k^+(0)<0\}=q+2p-1,\quad
\#\{\lambda_k^+(0)=0\}=3.
$$
\end{proposition}
Let
$$
\lambda_0^{(1)}\leqslant\lambda_1^{(1)}\leqslant\ldots\leqslant\lambda_k^{(1)}\leqslant\ldots\nearrow+\infty
\quad\text{and}\quad
\lambda_0^{(2)}\leqslant\lambda_1^{(2)}\leqslant\ldots\leqslant\lambda_k^{(2)}\leqslant\ldots\nearrow+\infty
$$
be the eigenvalues of the problems~\eqref{eq:l0-h1} and~\eqref{eq:l0-h2}
respectively.
As in the paragraph after Proposition~\ref{pr:sep}, for even $q$ denote by
$\lambda_k^{(1)+}$ and $\lambda_k^{(2)+}$ the $k$-th, $k\geqslant 0$, periodic
eigenvalues of the problems~\eqref{eq:l0-h1} and~\eqref{eq:l0-h2} respectively
(that is, those eigenvalues for which the corresponding eigenfunction is
$\frac{t_0}2$-periodic).

\begin{proposition}\label{pr:l0-h1}
We have
$$
\#\{\lambda_k^{(1)}<0\}=4p-1,\quad
\#\{\lambda_k^{(1)}=0\}=2.
$$
If $q$ is even, then, in addition, we have
$$
\#\{\lambda_k^{(1)+}<0\}=2p-1,\quad
\#\{\lambda_k^{(1)+}=0\}=2.
$$
\end{proposition}

\begin{proof}
It follows from Proposition~\ref{pr:esec} that the functions $\cos\phi\cos
2\theta$ and $\cos\phi\sin 2\theta$ are the eigenfunctions of the
problem~\eqref{eq:l0-h1} with eigenvalue~0.
These functions are linearly independent and
by~\cite[Proposition~4]{karpukhin2014spectral} both have exactly $4p$ zeros on
$[0,t_0)$. The Periodic Sturm Oscillation Theorem (see, for
example,~\cite[Theorem~3.1 in Chapter~8]{coddington1995theory}) implies that the
only zero eigenvalues of the problem~\eqref{eq:l0-h1} are $\lambda_{4p-1}^{(1)}$
and $\lambda_{4p}^{(1)}$.
This proves the first part of the proposition.
The second part follows from the fact that for even $q$ both mentioned
$0$-eigenfunctions are $\frac{t_0}2$-periodic.
\end{proof}

Unfortunately, for the problem~\eqref{eq:l0-h2} Proposition~\ref{pr:esec} gives
only one $0$-eigenfunction.
This means that we cannot determine the number of this eigenvalue using only the
Periodic Sturm Oscillation Theorem.
To override this difficulty, consider the following problem on a short interval
\begin{equation}\label{eq:l0-h2short}
\begin{cases}
-(4\pi^2\cos^2\phi\,h_2')'+(4\pi^2\dot\phi^2-2-8\pi^2\sin^2\phi\cos^2\phi\,\dot\theta^2)h_2=\lambda h_1,\\
h_2(T)=-h_2(0),\quad h_2'(T)=-h_2'(0).
\end{cases}
\end{equation}
Let
$$
\lambda_1^{(2')}\leqslant\lambda_2^{(2')}\leqslant\ldots\leqslant\lambda_k^{(2')}\leqslant\ldots\nearrow+\infty
$$
be the eigenvalues of this problem.

\begin{proposition}\label{pr:l0-h2aux}
We have $\lambda_1^{(2')}<0$ and $\lambda_2^{(2')}=0$.
\end{proposition}

\begin{proof}
Consider the substitution
$$
\tilde h_2(t)=\sqrt{p(t)}h_2(t)=2\pi\cos\phi(t)h_2(t).
$$
It is well-known that this substitution transforms the first equation
of~\eqref{eq:l0-h2short} to an equation of the form
$$
-\tilde h_2''(t)+q(t)\tilde h_2(t)=\lambda w(t)\tilde h_2(t),
$$
where $w(t)=\frac{1}{p(t)}$ (this is so-called Hill differential equation).
Let us figure out what is $q(t)$.
From Proposition~\ref{pr:esec}(ii) we know that the function
$2\pi\cos^4\phi\,\dot\theta$ solves this equation with $\lambda=-2$.
But this function is a nonzero constant (see~\eqref{eq:ots-dottheta}), that
is, we actually have $q(t)=-2w(t)$.
Thus the problem~\eqref{eq:l0-h2short} becomes
\begin{equation}\label{eq:l0-h2hill}
\begin{cases}
-\tilde h_2''(t)=(\lambda+2)w(t)\tilde h_2(t),\\
\tilde h_2(T)=-\tilde h_2(0),\quad\tilde h_2'(T)=-\tilde h_2'(0).
\end{cases}
\end{equation}
We now need the following easy claim, which is a continuous version of the
'permutation inequality'.
\begin{claim}\label{cl:l0-permineq}
Let $f(t)$ and $g(t)$ be two strictly decreasing functions on $[0,\tau]$. Then
$$
\int_0^\tau f(t)g(t)dt>\int_0^\tau f(t)g(\tau-t)dt.
$$
\end{claim}

\begin{proof}
Put $f^*(t)=f(t)-f(\tau-t)$ and define $g^*$ similarly.
Then $f^*$ and $g^*$ are both positive on $[0,\tau]$ and
$$
\int_0^\tau f(t)g^*(t)dt= \int_0^{\tau/2} f^*(t)g^*(t)>0,
$$
which proves the claim.
\end{proof}

Now from Proposition~\ref{pr:esec}(i) we know that the function
$\eta(t):=2\pi\cos^3\phi\,\dot\phi$ solves the problem~\eqref{eq:l0-h2hill} with
$\lambda=0$.
By the variational characterization of the eigenvalues of the
problem~\eqref{eq:l0-h2hill} with the function $\eta(\frac{T}2-t)$ as a test
function, we find
$$
\lambda_1^{(2')}+2\leqslant
\frac{\int_0^T \eta'(\frac{T}2-t)^2 dt}{\int_0^T w(t)\eta(\frac{T}2-t)^2 dt}=
\frac{\int_0^{T/2} \eta'(t)^2 dt}{\int_0^{T/2} w(t)\eta(\frac{T}2-t)^2 dt}<
\frac{\int_0^{T/2} \eta'(t)^2 dt}{\int_0^{T/2} w(t)\eta(t)^2 dt}=
\frac{\int_0^T \eta'(t)^2 dt}{\int_0^T w(t)\eta(t)^2 dt}=2,
$$
where Claim~\ref{cl:l0-permineq} is applied with $\tau=T/2$, $f(t)=w(t)$, and
$g(t)=\eta(\frac{T}2-t)^2$.
\end{proof}

\begin{remark}
The trick with Claim~\ref{cl:l0-permineq}, used in the proof of
Proposition~\ref{pr:l0-h2aux}, can be also used to simplify the proof
in~\cite[sec.~3.4]{karpukhin2014spectral}.
\end{remark}

\begin{proposition}\label{pr:l0-h2}
We have
$$
\#\{\lambda_k^{(2)}<0\}=2q,\quad
\#\{\lambda_k^{(2)}=0\}=1.
$$
If $q$ is even, then, in addition, we have
$$
\#\{\lambda_k^{(2)+}<0\}=q,\quad
\#\{\lambda_k^{(2)+}=0\}=1.
$$
\end{proposition}

\begin{proof}
Actually, this follows from~\cite[Proposition~11]{karpukhin2014spectral}), but
let us give a proof here for the sake of completeness.
Consider the $\lambda_1^{(2')}$- and $\lambda_2^{(2')}$-eigenfunctions of the
problem~\eqref{eq:l0-h2short}.
By Periodic Sturm Oscillation Theorem, both these eigenfunctions have one zero
on $[0,T)$.
Moreover, these eigenfunctions extend by antiperiodicity to the eigenfunctions
of the problem~\eqref{eq:l0-h2} with $2q$ zeroes.
Again by Periodic Sturm Oscillation Theorem and Proposition~\ref{pr:l0-h2aux} we
obtain
$\lambda_{2q-1}^{(2)}=\lambda_1^{(2')}<0$ and
$\lambda_{2q}^{(2)}=\lambda_2^{(2')}=0$, which proves the first claim.
The second claim follows from the fact that for even $q$ both eigenfunctions
extend to $\frac{t_0}2$-periodic eigenfunctions of~\eqref{eq:l0-h2}.
\end{proof}

\begin{proof}[Proof of Proposition~\ref{pr:l0}\nopunct]
follows from Propositions~\ref{pr:l0-h1} and~\ref{pr:l0-h2}.
\end{proof}

\subsection{Cases $l=1$ and $l=2$}\label{sec:l12}

This case is much more complicated because now the matrix
problem~\ref{eq:sep-sl} does not split into scalar problems. Our aim is the
following.

\begin{proposition}\label{pr:l12}
For each $b$ sufficiently close to zero, we have
\begin{align*}
2q+2p-1\leqslant&\#\{\lambda_k(1)<0\}\leqslant 4q-2,\quad
&2\leqslant&\#\{\lambda_k(1)=0\}\leqslant 4,\\
&\#\{\lambda_k(2)<0\}=0,\quad
&&\#\{\lambda_k(2)=0\}=1.
\end{align*}
If, in addition, $q$ is even, then
\begin{align*}
q+p-1\leqslant&\#\{\lambda_k^-(1)<0\}\leqslant 2q-2,\quad
&2\leqslant&\#\{\lambda_k^-(1)=0\}\leqslant 4,\\
&\#\{\lambda_k^+(2)<0\}=0,\quad
&&\#\{\lambda_k^+(2)=0\}=1.
\end{align*}
\end{proposition}

All the coefficients of the problem~\eqref{eq:sep-sl} are either $T$-periodic or
$T$-antiperiodic (recall that $T=t_0/2q$).
Thus it makes sense to consider the following matrix Sturm-Liouville problem
with quasiperiodic boundary conditions:
\begin{equation}\label{eq:short-sl}
\begin{cases}
S_l\h=\lambda\h,\\
\h(T)=\omega J\h(0),\quad \h'(T)=\omega J\h'(0),
\end{cases}
\end{equation}
where $J=\left(\begin{smallmatrix}-1 & 0 \\ 0 & 1\end{smallmatrix}\right)$,
$\omega$ is some $2q$-th root of unity, and $\h(t)$ is now a
\emph{complex}-valued vector function.
Let
$$
\lambda_1^{[\omega]}(l)\leqslant\lambda_2^{[\omega]}(l)\leqslant\ldots\leqslant
\lambda_k^{[\omega]}(l)\leqslant\ldots\nearrow+\infty
$$
be the eigenvalues of the problem~\eqref{eq:short-sl}.

\begin{proposition}\label{pr:l12-short}
Denote $\omega_0=e^{\pi i/q}$. Then
$$
\{\lambda_k(l)\mid k\geqslant 1\}=
\bigcup_{r=0}^{2q-1}\{\lambda_k^{[\omega_0^r]}(l)\mid k\geqslant 1\},
$$
and if $q$ is even, then
$$
\{\lambda_k^+(l)\mid k\geqslant 1\}=
\bigcup_{r=0}^{q-1}\{\lambda_k^{[\omega_0^{2r}]}(l)\mid k\geqslant 1\},\quad
\{\lambda_k^-(l)\mid k\geqslant 1\}=
\bigcup_{r=0}^{q-1}\{\lambda_k^{[\omega_0^{2r+1}]}(l)\mid k\geqslant 1\},
$$
where all equalities are equalities of multisets, i.e., each eigenvalue is
counted with its multiplicity.
\end{proposition}

\begin{proof}
The argument is almost the same as in the proof of the second part of
Proposition~\ref{pr:sep}.
In this case the operator $S_l$ commutes with the map $\sigma\colon \h(t)\mapsto
J\h(t+T)$.
Hence there exists a joint basis of eigenfunctions for $S_l$ and $\sigma$. Since
$\sigma^{2q}=\id$, we obtain that each $\sigma$-eigenvalue $\omega$ satisfies
$\omega^{2q}=1$ and hence equals $\omega_0^r$ for some $r=0,\ldots,2q-1$.
This proves the first claim.
For the second claim note that if $q$ is even, then an eigenfunction
of~\eqref{eq:short-sl} is $\frac{t_0}2$-periodic if $\omega^q=1$ and
$\frac{t_0}2$-antiperiodic if $\omega^q=-1$.
\end{proof}

Consider the quadratic form, associated with the Sturm-Liouville operator $S_l$
on the segment $[0,T]$,
$$
Q_l[\h]=\int_0^T (p(t)|\h'(t)|^2+\langle\h(t),A(l,t)\h(t)\rangle)dt=
\langle\h(t),p(t)\h'(t)\rangle\Bigr|_0^T+\int_0^T \langle \h(t),S_l\h(t)\rangle dt,
$$
where $\langle\cdot\,,\cdot\rangle$ is the standard Hermitian product in $\C^2$
(linear in first argument).
Let $\Sigma_l$ be the space of solutions of the equation $S_l\h=0$ on the
segment $[0,T]$.
Consider the map
$$
\Sigma_l\to\C^2\oplus\C^2,\quad
\h(t)\mapsto (\h(0),\h(T)).
$$
Since $\dim\Sigma_l=4$, this map is an isomorphism provided that it is injective
or, equivalently, that zero is not an eigenvalue of the Dirichlet problem
\begin{equation}\label{eq:short-dir}
\begin{cases}
S_l\h=\mu\h,\\
\h(0)=\mathbf 0,\quad \h(T)=\mathbf 0.
\end{cases}
\end{equation}
In this case one can define the following quadratic form $\alpha_l$ on
$\C^2\oplus\C^2$
\begin{equation}\label{eq:l12-alpha}
\alpha_l[(\mathbf v_0,\mathbf v_T)]=Q_l[\g]=\langle\g(t),p(t)\g'(t)\rangle\Bigr|_0^T,
\end{equation}
where $\mathbf v_0,\mathbf v_T\in\C^2$ and $\g$ is a unique element of
$\Sigma_l$ such that $\g(0)=\mathbf v_0,\g(T)=\mathbf v_T$.
Let $\mathbf e_1,\mathbf e_2$ be the standard basis of $\C^2$.

The following proposition is our main tool in the study of cases $l=1$ and
$l=2$.
It is a direct consequence of~\cite[Proposition~2.6]{edwards1964generalized}.
Let
$$
\mu_1(l)\leqslant\mu_2(l)\leqslant\ldots\leqslant\mu_k(l)\leqslant\ldots\nearrow+\infty
$$
be the eigenvalues of the problem~\eqref{eq:short-dir}.

\begin{proposition}\label{pr:l12-edw}
Suppose that zero is not an eigenvalue of the problem~\eqref{eq:short-dir}.
Then
\begin{gather*}
\#\{\lambda_k^{[\omega]}(l)<0\}=\#\{\mu_k(l)<0\}+i_l(\omega),\\
\#\{\lambda_k^{[\omega]}(l)=0\}=n_l(\omega),
\end{gather*}
where $i_l(\omega)$ and $n_l(\omega)$ denote, respectively, the index and the
nullity of the quadratic form $\alpha_l$ restricted to a subspace of
$\C^2\oplus\C^2$ spanned by $(\e_1,-\omega\e_1)$ and $(\e_2,\omega\e_2)$.
\end{proposition}

Let $\Gr_l(\omega)$ be the Gram matrix of the restriction of $\alpha_l$ on
$\mathrm{span}\{(\e_1,-\omega\e_1),(\e_2,\omega\e_2)\}$.

\begin{proposition}\label{pr:l12-detgr}
For each $l$ the determinant of $\Gr_l(\omega)$ is a polynomial in $\Re\omega$
with real coefficients of degree at most~$2$.
\end{proposition}

\begin{proof}
Consider the following basis in $\C^2\oplus\C^2\simeq\C^4$:
$$
\epsilon_1=(\e_1,\mathbf 0),\quad
\epsilon_2=(\e_2,\mathbf 0),\quad
\epsilon_3=(\mathbf 0,\e_2),\quad
\epsilon_4=(\mathbf 0,\e_1).
$$
Then $(\e_1,-\omega\e_1)=\epsilon_1-\omega\epsilon_4$ and
$(\e_2,\omega\e_2)=\epsilon_2+\omega\epsilon_3$.
We have
$$
\Gr_l(\omega)=\begin{pmatrix}
a_{11}-\omega a_{41}-\overline\omega a_{14}+a_{44} &
a_{12}-\omega a_{42}+\overline\omega a_{13}-a_{43}\\
a_{21}+\omega a_{31}-\overline\omega a_{24}-a_{34} &
a_{22}+\omega a_{32}+\overline\omega a_{23}+a_{33}
\end{pmatrix},
$$
where $a_{ij}=\alpha_l[\epsilon_i,\epsilon_j]\;(i,j=1,2,3,4)$.
Note that $a_{ij}\in\R$ and $a_{ij}=\overline{a_{ji}}=a_{ji}$ since $S_l$ is a
real operator and the form $\alpha_l$ is Hermitian.
Finally, $a_{ij}=a_{5-i,5-j}$ since the operator $S_l$ is invariant under the
map $t\mapsto T-t$. With all these relations we get
$$
\Gr_l(\omega)=\begin{pmatrix}
2a_{11}-2a_{14}\Re\omega &
-2ia_{13}\Im\omega\\
2ia_{13}\Im\omega &
2a_{22}+2a_{23}\Re\omega
\end{pmatrix}.
$$
Since $|\omega|=1$, the proposition follows.
\end{proof}

Our next step is to consider the limit case $b\to 0$.
In this case we have
$$
\phi(t)\equiv 0,\quad
\theta(t)=\frac{t}{2\pi},\quad
T=\lim_{b\to 0}T(b)=2\pi^2,
$$
and
$$
p(t)\equiv 4\pi^2,\quad
A(l,t)=\begin{pmatrix}
l^2-4 & 0\\
0 & l^2-2
\end{pmatrix}.
$$
Thus the operator $S_l$ splits into a pair of scalar second-order differential
operators with constant coefficients.
In particular, $\mu_k(l)$ and $\Gr_l(\omega)$ can be calculated explicitly.

\begin{proposition}
Let $b=0$. Then
\begin{equation}\label{eq:l12-mub0}
\mu_1(1)=-1,\,\mu_2(1)=1>0,\quad\mu_1(2)=2>0,
\end{equation}
and
\begin{align}
\Gr_1(\omega)&=\begin{pmatrix}
\frac{4\sqrt{3}\pi}{\sin\frac{\sqrt{6}\pi}2}(\cos\frac{\sqrt{6}\pi}2+\Re\omega) & 0\\
0 & \frac{4\pi}{\sin\frac{\sqrt{2}\pi}2}(\cos\frac{\sqrt{2}\pi}2-\Re\omega)
\end{pmatrix},\label{eq:l12-b0gr1}\\
\Gr_2(\omega)&=\begin{pmatrix}
4\sqrt{2}(1+\Re\omega) & 0\\
0 & \frac{4\sqrt{2}\pi}{\sinh\pi}(\cosh\pi-\Re\omega)
\end{pmatrix}.\label{eq:l12-b0gr2}
\end{align}
\end{proposition}

\begin{proof}
This is a tedious but very straightforward computation.
Let us compute, for example, the first entry of the matrix $\Gr_1(\omega)$.
This entry equals $\alpha_l[(\e_1,-\omega\e_1)]=Q_1[\g]$, where
$$
\g(t)=\frac{1}{\sin\frac{\sqrt 6\pi}2}\begin{pmatrix}
\sin\frac{\sqrt 3(\sqrt 2\pi^2-t)}{2\pi}-\omega\sin\frac{\sqrt 3t}{2\pi} \\ 0
\end{pmatrix}
$$
is a unique solution of the equation $S_l\h=\mathbf 0$ such that
$\g(0)=\e_1,\g(T)=-\omega\e_1$.
We have
$$
Q_1[\g]=\langle\g(t),p(t)\g'(t)\rangle\Bigr|_0^T=
\langle\g(t),4\pi^2\g'(t)\rangle\Bigl|_0^{\sqrt 2\pi^2}=
\frac{4\sqrt{3}\pi}{\sin\frac{\sqrt{6}\pi}2}(\cos\frac{\sqrt{6}\pi}2+\Re\omega).
$$
All other entries of $\Gr_1(\omega)$ and $\Gr_2(\omega)$ are computed similarly.
\end{proof}

In order to pass to the limit $b\to 0$, we need he following technical
proposition.

\begin{proposition}\label{pr:l12-cont}
Fix $l,\omega$ and consider the values $\mu_k(l)$ and (the entries of) the
matrix $\Gr_l(\omega)$ as functions in $b$.
Then these functions are continuous at $b=0$.
\end{proposition}

\begin{proof}
To show continuity of the matrix $\Gr_l(\omega)$ it suffices to establish the
continuity of $a_{ij}$ from the proof of Proposition~\ref{pr:l12-detgr}.
This is essentially the continuity of the solution of a system of ODEs on the
parameter.
Indeed, by~\eqref{eq:l12-alpha} it suffices to check the continuous dependence
on $b$ of $\g'(0)$ and $\g'(T)$, where $\g$ is a unique solution of the equation
$S_l\g=0$ such that $\g(0)=\epsilon_i,\g(T)=\epsilon_j$.
For each $b$ consider a fundamental system of solutions $\h^i(t)$ ($i=1,2,3,4$)
of the equation $S_l\h=0$ with fixed (i.e., independent of $b$) initial
conditions at $t=0$.
Then $\g(t)$ is a linear combination of these fundamental solutions.
Moreover, the coefficients of this linear combinations are rational functions in
$\h^i(T)$.
Thus, the continuity of $\g'(0)$ and $\g'(T)$ follows from the continuity of
$\h^i(T),(\h^i)'(T)$.
This last one is exactly continuous dependence of the solution on the parameter;
see, for example,~\cite[Theorem~7.4 in Chapter~1]{coddington1995theory}.
The continuity of $\mu_k(l)$ follows from~\cite[Theorem~6.1]{hu2019singularity}.
\end{proof}

\begin{proof}[Proof of Proposition~\ref{pr:l12}]
\textit{Case $l=1$.} Let $b$ be sufficiently close to zero.
Then by~\eqref{eq:l12-mub0} and continuity of $\mu_k(1)$ at $b=0$ we have
$\mu_1(1)<0$ and $\mu_2(1)>0$ so that
\begin{equation}\label{eq:l12-l1mu}
\#\{\mu_k(1)<0\}=1.
\end{equation}
Recall from Proposition~\ref{pr:l12-detgr} that $\det\Gr_1(\omega)$ is a
quadratic polynomial in $\Re\omega$ with real coefficients.
For $b=0$, the roots of this polynomial are $\cos\frac{\sqrt 2\pi}2$ and
$-\cos\frac{\sqrt 6\pi}2$ by~\eqref{eq:l12-b0gr1}.
By continuity of $\Gr_1(\omega)$ at $b=0$ (Proposition~\ref{pr:l12-cont}), the
polynomial $\det\Gr_1$ has two real roots $s_1$ and $s_2$, which are close to
$\cos\frac{\sqrt 2\pi}2$ and $-\cos\frac{\sqrt 6\pi}2$ respectively.
Further, it follows from Proposition~\ref{pr:esec}(i) that the complex-valued
vector function
$$
e^{i\theta}\begin{pmatrix}
\sin\phi\\
2\pi\cos\phi(\sin\phi\,\dot\phi-i\cos\phi\,\dot\theta)
\end{pmatrix}
$$
is a vector eigenfunction of the problem~\eqref{eq:short-sl} with
$l=1,\lambda=0,\omega=\omega_0^p$, where $\omega_0=e^\frac{\pi i}q$.
This means that in fact
\begin{equation}\label{eq:l12-l1s1}
s_1=\Re\omega_0^p=\Re\omega_0^{-p}.
\end{equation}
Also note that $|\cos\frac{\sqrt 2\pi}2|<-\cos\frac{\sqrt 6\pi}2$ and hence
\begin{equation}\label{eq:l12-l1s1s2}
s_2>|s_1|=\Re\omega_0^{q-p}\quad\text{(for $b$ close to zero)}.
\end{equation}
In the notation of Propositions~\ref{pr:l12-short} and~\ref{pr:l12-edw} we have
\begin{equation}\label{eq:l12-l1in}
i_1(\omega)=\begin{cases}
0,&\Re\omega\in[-1,s_1],\\
1,&\Re\omega\in(s_1,s_2],\\
2,&\Re\omega\in(s_2,1],
\end{cases}\quad
n_1(\omega)=\begin{cases}
1,&\Re\omega=s_1,s_2,\\
0,&\text{otherwise},
\end{cases}
\end{equation}
and
\begin{align*}
\#\{\lambda_k(1)<0\}&=\sum_{r=0}^{2q-1}\#\{\lambda_k^{[\omega_0^r]}(1)<0\}
&&\text{(by Proposition~\ref{pr:l12-short})}\\
&=\sum_{r=0}^{2q-1}\#\{\mu_k(1)<0\}+i_1(\omega_0^r)
&&\text{(by Proposition~\ref{pr:l12-edw})}\\
&=2q+\#\{r\mid\Re\omega_0^r\in(s_1,1]\}+\#\{r\mid\Re\omega_0^r\in(s_2,1]\}
&&\text{(by~\eqref{eq:l12-l1mu} and~\eqref{eq:l12-l1in})}\\
&=2q+2p-1+\#\{r\mid\Re\omega_0^r\in(s_2,1]\}
&&\text{(by~\eqref{eq:l12-l1s1})}\\
&\in [2q+2p-1,4q-1],
&&\text{(by~\eqref{eq:l12-l1s1s2})},
\end{align*}
and similarly
$$
\#\{\lambda_k(1)=0\}=
\sum_{r=0}^{2q-1}\#\{\lambda_k^{[\omega_0^r]}(1)=0\}=
\sum_{r=0}^{2q-1} n_1(\omega_0^r)=
\#\{r\mid\Re\omega_0^r=s_1\text{ or }s_2\}\in [2,4].
$$
For even $q$ the computations are similar and therefore are omitted.

\textit{Case $l=2$} is similar but easier. By~\eqref{eq:l12-mub0} and continuity
of $\mu_k(2)$ at $b=0$ we have $\mu_1(2)>0$ so that
$$
\#\{\mu_k(2)<0\}=0.
$$
For $b=0$, the roots of this polynomial are $-1$ and $\cosh\pi$
by~\eqref{eq:l12-b0gr2}.
By continuity of $\Gr_2(\omega)$ at $b=0$ (Proposition~\ref{pr:l12-cont}), the
polynomial $\det\Gr_2$ has two real roots $s_1$ and $s_2$, which are close to
$-1$ and $\cosh\pi$ respectively.
It follows from Proposition~\ref{pr:esec}(i) that the vector function
$$
\begin{pmatrix}
\cos\phi\\
2\pi\cos^2\phi\,\dot\phi
\end{pmatrix}
$$
is a vector eigenfunction of the problem~\eqref{eq:short-sl} with
$l=2,\lambda=0,\omega=-1$.
This means that in fact $s_1=-1$. Since the second root $s_2$ is close to
$\cosh\pi>1$, we obtain
$$
i_2(\omega)=0,\quad
n_2(\omega)=\begin{cases}
1,&\omega=-1,\\
0,&\text{otherwise},
\end{cases}
$$
and finally
\begin{align*}
\#\{\lambda_k(2)<0\}&=
\sum_{r=0}^{2q-1}\#\{\lambda_k^{[\omega_0^r]}(2)<0\}=
\sum_{r=0}^{2q-1}\#\{\mu_k(2)<0\}+i_2(\omega_0^r)=0,\\
\#\{\lambda_k(2)=0\}&=
\sum_{r=0}^{2q-1}\#\{\lambda_k^{[\omega_0^r]}(2)<0\}=
\sum_{r=0}^{2q-1} n_2(\omega_0^r)=1.
\end{align*}
Again, we omit the case of even $q$ since it is similar.
\end{proof}

\smallskip

\begin{proof}[Proof of Theorem~\ref{th:intro-main}]
Let $b$ be small enough as required in Proposition~\ref{pr:l12}.
Let $q$ be odd. By Propositions~\ref{pr:sep-ind},~\ref{pr:sep-l3},~\ref{pr:l0},
and~\ref{pr:l12} we have
\begin{align*}
\Ind(\widetilde O_{p,q})&=\#\{\lambda_k(0)<0\}+2\#\{\lambda_k(1)<0\}+2\#\{\lambda_k(2)<0\}\in [6q+8p-3,10q+4p-5],\\
\Nul(\widetilde O_{p,q})&=\#\{\lambda_k(0)=0\}+2\#\{\lambda_k(1)=0\}+2\#\{\lambda_k(2)=0\}\in [9,13].
\end{align*}
If $q$ is even, then, similarly,
\begin{align*}
\Ind(\widetilde O_{p,q})&=\#\{\lambda_k^+(0)<0\}+2\#\{\lambda_k^-(1)<0\}+2\#\{\lambda_k^+(2)<0\}\in [3q+4p-3,5q+2p-5],\\
\Nul(\widetilde O_{p,q})&=\#\{\lambda_k^+(0)=0\}+2\#\{\lambda_k^-(1)=0\}+2\#\{\lambda_k^+(2)=0\}\in [9,13].
\end{align*}
\end{proof}

\begin{remark}[Discussion]\label{rem:l12-disc}
It seems that any exact computation of $\Ind(\widetilde O_{p/q})$ for $b$ close
to zero should contain a very accurate control over the root $s_2$ from the
proof of Proposition~\ref{pr:l12}.
This root (in contrast to $s_1$, which comes from a Jacobi field) has not any
clear geometric meaning and hardly can be calculated exactly.
However, numerical experiments in Wolfram Mathematica suggest that when $b$
varies from~0 to $-\pi/2$, the root $s_2$ increases and passes through~1 at some
moment.
More precisely, we have the following conjecture.

\begin{conjecture}\label{con:l12-disc}
There exists $r_0\in\left(2/3,\sqrt 2/2\right)$ such that for each
$p/q\in\left(1/2,r_0\right)$ one has
$$
\Ind(\widetilde O_{p/q})=\begin{cases}
6q+8p-3,&\text{$q$ is odd;}\\
3q+4p-3,&\text{$q$ is even;}
\end{cases}
\quad
\Nul(\widetilde O_{p/q})=9.
$$
In particular,
$$
\Ind(\widetilde O_{2/3})=31,\quad
\Nul(\widetilde O_{2/3})=9.
$$
\end{conjecture}

One can try to verify this conjecture at least for $b$ close to $-\pi/2$ using
the same method as in the proof of Theorem~\ref{th:intro-main}.
The problem here is that the limit matrix Sturm-Liouville problem in this case
is somewhat very singular: the function $p(t)$ vanishes at the ends of the
interval and the matrix potential $A(l,t)$ contains a Dirac delta function.
It is not clear how one should treat the eigenvalues of this problem.

Finally, let us note that similar difficulties appear when one tries to compute
the indices of some similar codimension~2 minimal surfaces.
For example, the indices of bipolar Lawson $\tau$-surfaces $\tilde\tau_{m,k}$ in
$\Sph^4$ and Fraser-Sargent free boundary minimal surfaces $IFS_{k,l}$ in
$\mathbb B^4$ are known only for $m=3,k=1$ (see~\cite[proof of
Corollary~6.4]{karpukhin2021stability}) and $k=2,l=1$
(see~\cite{medvedev2023critical,medvedev2023index}) respectively.
\end{remark}

\bibliography{English}
\bibliographystyle{alpha}

\end{document}